\documentclass[12pt]{article}

\usepackage[margin=1in]{geometry}  
\usepackage{graphicx}              
\usepackage{amsmath}               
\usepackage{amsfonts}              
\usepackage{amsthm}                
\usepackage{tikz}

\newtheorem{thm}{Theorem}[section]
\newtheorem{lem}[thm]{Lemma}

\newtheorem{cor}[thm]{Corollary}

\newtheorem{exm}[thm]{Example}

\newcommand{\layer}[2]{{^{[#2]}#1}}

\theoremstyle{definition}
\newtheorem{defn}[thm]{Definition}

\newtheorem{exmp}[thm]{Example}

\title{Factorization of polynomials in supertropical algebra}
\author{Erez Sheiner}

\begin{document}

\maketitle

\begin{abstract}
Tropical algebraic geometry is a degeneration of classical geometry which loose the property of unique factorization for polynomials. In this paper we explore a structure that is known to be a semi-degeneration between the classical algebra and the tropical algebra, and show that unique factorization fails in several variables due to geometric reasons, not just algebraic. We also show that unique factorization does hold for a certain interesting subset of polynomials.\\
\end{abstract}

\tableofcontents

\section{Introduction}

We wish to study algebraic geometry over the max-plus or tropical algebra. In order to do so, we should look at roots of polynomials.\\

The common definition for tropical roots of polynomials is a point of equality between two monomials or more (ref. \cite{M1},\cite{GA}). Izhakian has built a structure that defines the sum of two equal elements as a ``ghost" element, and treats these ghosts as zeros since we want to view them as roots.\\

Further research led Izhakian and Rowen (\cite{IR2}) to the idea of a graded algebra. Not only do we ``remember" the sum as a ``ghost", we also keep a layer element that gives us more information. For example, assuming a natural or a tangible element is of layer one, then the sum of three tangible elements is of layer three. In the broad perspective we will see the graded algebra as a lesser degeneration of the classical geometry than the supertropical algebra, which is lesser than tropical geometry.\\

We introduce an extension of the max-plus algebra with layers, called exploded layered tropical algebra (or ELT algebra for short). Given this new structure we further refine the definition of a root to be a point where the layer of the evaluation of the polynomial is zero. This structure is similar to Parker's ``exploded" semiring and holomorphic curves (\cite{PR}). Parker uses exploded manifolds to define and compute Gromov-Witten invariants.\\

\begin{exm}
The roots of the polynomial $f(\lambda)=\layer{0}{a}\lambda^2+\layer{2}{1}$ in ELT algebra over the complex field are $\layer{1}{\frac{\pm i}{\sqrt{a}}}.$\\

Indeed,
$$f(\layer{1}{\frac{\pm i}{\sqrt{a}}}) = \layer{0}{a}\Big(\layer{1}{\frac{\pm i}{\sqrt{a}}}\Big)^2+\layer{2}{1}=
\layer{0}{a}\Big(\layer{2}{\frac{-1}{a}}\Big)+\layer{2}{1}=\layer{2}{-1}+\layer{2}{1}=\layer{2}{0}.
$$

\end{exm}

In this paper, we focus on the problem of factorization of polynomials and the necessary requirements from the layer structure. The factorization of tropical polynomials is important for a number of reasons: first, the factors of a certain polynomial help us split the variety of the polynomial into smaller varieties; second, the way polynomials factor affects the algebraic structure of the polynomials ring and its ideals. This is important since we aim to create an extensive algebraic base. Also, Gathmann (ref. \cite{GA}) explained the importance of factorization of tropical polynomials and its connection to ordinary polynomial factorization.\\

We will show that in this structure most polynomials in one indeterminate factor uniquely. We will also show that polynomials in several variables, in which all monomials have the same tangible value at some point (called primary polynomials), factor uniquely.\\

\begin{exmp}
For example, consider the variety of three geometric lines which intersect at (0,0).\\
$$f=(x+y)(x+0)(y+0)=(x+y+0)(xy+x+y).$$
This variety may factor into the three geometric lines, or a tropical line and a tropical quadratic factor. The distinction between these two cases is encapsulated in the layer of the intersection point. We will see that since this is a primary polynomial it factors uniquely in our structure.\\
\end{exmp}

\subsection{Exploded Layered Tropical Algebra}

Next we define the algebraic structure we use throughout this paper, which is inherited from the work of Parker (\cite{PR}) and Izhakian and Rowen (\cite{IR1},\cite{IR2}).\\

\begin{defn}
Let $L$ be a set closed under addition and multiplication and $F$ a totally ordered group (such as $(\mathbb{R},+)$ or $(\mathbb{Q},+)$). An \textit{ELT algebra} is the set $\{\layer{\lambda}{\ell}|\lambda\in F,\ell\in L\}$ together with the semi-ring structure:\\

\begin{enumerate}
    \item $\layer{\lambda}{\ell_1}+\layer{\lambda}{\ell_2} = \layer{\lambda}{\ell_1+_L\ell_2}, $
    \item If $\lambda_1>\lambda_2$ then $\layer{\lambda_1}{\ell_1}+\layer{\lambda_2}{\ell_2} = \layer{\lambda_1}{\ell_1}, $
    \item $\layer{\lambda_1}{\ell_1}\cdot\layer{\lambda_2}{\ell_2}=\layer{(\lambda_1+_{F}\lambda_2)}{\ell_1\cdot_L\ell_2}. $
\end{enumerate}

\end{defn}

Let R be an ELT algebra. We write $s:R\rightarrow L$ for the function which extracts the coefficient:
$$s(\layer{\lambda}{\ell})=\ell,$$
and $t:R\rightarrow F$ for the function which extracts the tangible value:
$$t(\layer{\lambda}{\ell})=\lambda.$$

We extend the total order on $F$ to a partial order on $R$ in the natural way:
$$\layer{\lambda_1}{\ell_1}\geq \layer{\lambda_2}{\ell_2} \iff \lambda_1 \geq_F \lambda_2.$$

\begin{exm}
Zur Izhakian's supertropical geometry (ref. \cite{IZ}) is equivalent to an ELT algebra with $L=\{1,2\}$ such that
$$1+1=2,1+2=2,2+2=2$$
and
$$1\cdot 1=1, 1\cdot 2=2, 2\cdot 2 = 2.$$\\

The supertropical "ghost" element $1^\nu$ is equivalent to $\layer{1}{2}$ in the ELT notation, and the tangible element $1$ to $\layer{1}{1}$.\\

Therefore, in this paper we will refer to this specific ELT algebra as a supertropical algebra.\\
\end{exm}

\begin{exm}
The classical max-plus algebra is equivalent to the trivial ELT algebra with $L=\{1\}$. We call this case tropical algebra.
\end{exm}

We would also like to define the element $-\infty$. Define 
$$\overline{R}:=R\cup \{-\infty\}$$
such that for all $a\in \overline{R}$:
$$a+ (-\infty)=(-\infty)+a=a,$$
$$a\cdot (-\infty)=(-\infty)\cdot a=(-\infty).$$
We also define
$$s(-\infty):=0_\mathbb{F},$$
and
$$ELTrop(0_K)=-\infty.$$

\section{Basic Definitions}
We wish to define polynomials over the ELT algebra with $F=\mathbb{R}$. One must note that unlike polynomials over classical algebra, two ELT polynomials may be equal everywhere yet contain different monomials.

\begin{exm}
Consider the two ELT polynomials
$$f(x)=x^2+\layer{2}{1},$$
and,
$$g(x)=x^2+x+\layer{2}{1}.$$

For each $x\in R$ such that $x>1$, the monomial $x^2$ dominates the other monomial since $x^2>x,2$. Thus, in this case, $f(x)=g(x)=x^2$.\\

If $t(x)=1$ then $f(x)=x^2+\layer{2}{1}$. In the polynomial $g$, the monomials $x^2,\layer{2}{1}$ dominates the monomial $x$ and so $g(x)=f(x)$ as well.\\

The last case is $x<1$ in which similarly $f(x)=g(x)=\layer{2}{1}$.\\

Therefore $f$ and $g$ are equal at every point of $R$ even though they contain different monomials.\\
\end{exm}

For this reason we define ELT polynomials as \textbf{functions}.

\begin{defn}
An ELT \textit{polynomial} $p$ is a function $p:R^n\rightarrow R$ of the form
$$p(\lambda_1,...,\lambda_n)=\sum_{I\in G} a_I\lambda^I,$$
where $G\subseteq \mathbb{N}^n$ is a finite set and for all $I\in G$ the coefficient $a_I$ is in $R$.\\

We denote the set of all such polynomials as $R[\lambda_1,...,\lambda_n]$.
\end{defn}

\begin{defn}
Let $f$ be a polynomial of the form $f=\sum_{i=1}^{n}{h_i}$, where $h_i$ are monomials. Write $f_h=\sum_{h_i\neq h}h_i$. Then:
\begin{enumerate}
 \item A monomial $h$ is called \textbf{inessential at a point $a$} if $f_h(a)=f(a)$. If $h$ is inessential at every point, then $h$ is called \textbf{inessential}.
 \item A monomial $h$ is called \textbf{essential at a point $a$} if it is not inessential at $a$ (needed for layer zero) and $f(a)=h(a)$. If such a point exists, then $h$ is called \textbf{essential}.
 \item A monomial $h$ is called \textbf{quasi-essential at a point $a$} if it is neither essential nor inessential at $a$. If $h$ is neither essential nor inessential, then it is called \textbf{quasi-essential}.\\
\end{enumerate}
\end{defn}

\begin{exmp}
Consider $f=\lambda^2+\layer{1}{1}\lambda + \layer{2}{1}$. Then $\lambda^2$ is essential at each tangible point with value greater than $1$. At $1$ all of the monomials are quasi-essential, and for tangibles lesser than $1$ the fixed monomial $2$ is essential.\\
\end{exmp}

\begin{defn}
The monomial of a univariate polynomial $f$ with the highest power of $\lambda$ is called the \textit{leading} monomial. The monomial with the lowest power is called the \textit{tail} monomial. Any other monomial is called a \textit{middle} monomial.\\
\end{defn}

\begin{defn}

A \textit{corner root} of an ELT polynomial is a point at which at least two monomials dominate. i.e., $(c_1,...,c_n)\in R^n$ is a corner root of $p(\lambda_1,...,\lambda_n)=\sum_{I\in G} a_I\lambda^I$ if the set $\{I\in G|t(a_Ic^I)=t(p(c_1,...,c_n))\}$ is of order $2$ at least.\\

In other words, a point $a$ is a corner root of a polynomial $f=\sum_{i=1}^{n}{h_i}$ if $h_k$ is quasi-essential at $a$ for some $1\leq k\leq n$.\\
\end{defn}

\begin{lem}
A polynomial $f$ in one variable has only finitely many corner roots.\\
\end{lem}
\begin{proof}
By definition, a corner root is a point where two monomials $h_1,h_2$ have the same tangible value. There are only finitely many different possible pairs of monomials, each contributing at most one root.\\
\end{proof}

\begin{defn}

Let $f$ and $g$ be two polynomials in several variables. We say that $f$ and $g$ are \textit{root equivalent} if
$$s\Big(f(c)\Big)=s\Big(g(c)\Big),$$
for every corner root $c$.\\
\end{defn}

\begin{defn}
Let $f$ be a multivariate polynomial. If all of the monomials of $f$ have the same tangible value at the point $a$, then $f$ is called \textbf{primary in $a$}.\\
\end{defn}

\begin{lem}
Let $f$ be primary in $a$. Then $f$ is of the form $f=c\sum \layer{0}{b_I}a_I\lambda^I$ where $I=(i_1,...,i_n)\in (\mathbb{N}\cup \{0\})^n$, $a=(a_1,...,a_n)$, $a_I=a_1^{-i_1}\cdot\cdot\cdot a_n^{-i_n}$ , $b_I \in L$, $b_I\neq 0$, $\lambda^I=\lambda_1^{i_1}\cdot\cdot\cdot\lambda_n^{i_n}$, and $c$ is tangible.
\end{lem}

\begin{proof}
First we let $c$ be the tangible value of $f(a)$. Let $d\lambda_I$ be any monomial of $f$. Then it is quasi-essential or essential at $a$. Thus the tangible value of $da_1^{i_1}\cdot\cdot\cdot a_n^{i_n}$ must be $c$. Therefore, $d=\layer{0}{b_I}a_1^{-i_1}\cdot\cdot\cdot a_n^{-i_n}c$ where $b_I$ is the layer of $d$. $b_I\neq 0$ since otherwise the monomial could not be quasi-essential.
\end{proof}

\begin{defn}
The \textit{essential part} of a general polynomial $f=\sum_{i \in I} h_i$ at a point $a$ is the polynomial $$f_a = \sum_{k \in K} h_k$$ such that $k \in K \subseteq I$ if and only if $h_k$ is not inessential at $a$.\\
\end{defn}

It is fairly clear that the essential part at $a$ is always primary at $a$.\\

\section{Expanding the Structure of Layers}
We wish to obtain a basic algebraic result - unique factorization - for polynomials over the ELT structure. In their paper (ref. \cite{IR1}), Izhakian and Rowen showed that unique factorization fails in their original supertropical structure even when considering polynomials as functions.\\

As we explained in the introduction, we try to expand the structure of $L$. First we consider $L=\mathbb{N}$ with the usual operations. Considering $R[\lambda]$, the polynomials in one variable over this structure, we wish to know if there is unique factorization. A rather simple counterexample arises which we will explore. \\

We would hope generally that a polynomial would factor according to the variety of its roots; for each of the connected parts there would be one factor (not necessarily irreducible). In the one variable case, each such part is a single point.\\

Now we are ready to study the following polynomial over the supertropical algebra:
$$f=\lambda^2 + \layer{1}{2}\lambda + \layer{0}{1}.$$

In this example, the roots are $\lambda=1$ and $\lambda=-1$. We expect that the above polynomial will factor into $(\lambda+\layer{1}{2})(\layer{1}{2}\lambda+0)$, since these are the roots with the correct layers. We are rather close but not exactly there, as we get
$$(\lambda+\layer{1}{2})(\layer{1}{2}\lambda+0)= \layer{1}{2}\lambda^2 + 0\lambda + \layer{1}{2}\layer{1}{2}\lambda + \layer{1}{2} = \layer{1}{2}\lambda^2 + \layer{1}{2}\lambda + \layer{1}{2} = \layer{1}{2}(\lambda^2 + \layer{1}{2}\lambda + 0) = \layer{1}{2}f.$$
The polynomial $f$ is irreducible and also $(\lambda+\layer{1}{2})$ and $(\layer{1}{2}\lambda+0)$ are irreducible, which contradicts unique factorization of $g=\layer{1}{2}f=(\lambda+\layer{1}{2})(\layer{1}{2}\lambda+0)$.\\

However, these factorizations are not inherently different. The main problem is the lack of an inverse for the layer. If we add fractional positive layers, i.e. take $L=\mathbb{Q}^+$, the unique factorization of the polynomial will be

$$\layer{(-1)}{1/2}(\lambda+\layer{1}{2})(\layer{1}{2}\lambda+0) = (\lambda+\layer{1}{2})(\lambda+\layer{(-1)}{1/2}).$$

\section{Positive Rational Layers}
For this section we fix $L=\mathbb{Q}^+$.\\

\subsection{Monomial equality}

We will show that polynomials which are equal as functions must consist of the same monomials (other than the inessential ones). Moreover, we will see that given enough points of equality, two polynomials are equal everywhere.\\

\begin{lem}
Let $k>l \in \mathbb{N}$ and $a,b \in R$. The polynomial $f=a\lambda^k+b\lambda^l$ has a corner root $x\in R$. For any substitution $\lambda<x$, the second monomial is essential and for any substitution $\lambda>x$, the first monomial is essential.
\end{lem}

\begin{proof}
A tangible point $x$ is a corner root if $t(ax^k)=t(bx^l)$, and therefore $t(x^{k-l})=t(ba^{-1})$. Given our assumptions, such a corner root exists. If $y>x$ then $t(y^{k-l})>t(x^{k-l})=t(ba^{-1})$ and therefore $ay^k>by^l$; similarly if $y<x$ then $ay^k<by^l$. We see that the corner root is like a scale; on one side one monomial is essential and on the other side the second monomial is essential. This is the piecewise linear behavior of supertropical algebras.
\end{proof}

\begin{cor}
Given a polynomial in one variable
$$f=a_{r_1}\lambda^{r_1}+...+a_{r_n}\lambda^{r_n}$$
such that $r_1>r_2>...>r_n$, $n>2$ and $a_{r_i}\neq -\infty$, then $f$ has a finite set of corner roots
$$x_k>x_{k-1}>...>x_1, k>0.$$

The monomials $a_{r_1}\lambda^{r_1}+a_{r_2}\lambda^{r_2}+...+a_{r_i}\lambda^{r_i}$ are quasi-essential at $x_k$ for some $i>0$. The monomial $a_{r_1}\lambda^{r_1}$ is essential at $\lambda>x_k$, the monomial $a_{r_i}\lambda^{r_i}$ is essential at $x_k>\lambda>x_{k-1}$, and the monomials between them are inessential at $\lambda\neq x_k$.\\

The monomials $a_{r_i}\lambda^{r_i}+...+a_{r_j}\lambda^{r_j}$ are quasi-essential at $x_{k-1}$ for some $j>i$. The rightmost monomial is essential at $x_{k-1}>\lambda>x_{k-2}$, and the monomial between it and the leftmost monomial are inessential at $\lambda\neq x_{k-1}$.\\

This continues until the rightmost monomial is $a_{r_n}\lambda^{r_n}$ which is essential at $\lambda<x_1$.\\
\end{cor}

\begin{lem}
If $f=\sum h_k$ is a multivariate polynomial with a non-empty set of corner roots, then for all $i$, $h_i$ is quasi-essential on at least one corner root of $f$.
\end{lem}

\begin{proof}
Recall that $h_i$ is not inessential. Therefore, let $c$ be a point so that $h_i$ is essential or quasi-essential at $c$. If $c$ is a corner root we are done, so we assume that it is not. \\

The monomial $h_i$ is in several variables and is not inessential at some point $c=(c_1,...,c_n)$. Consider $f$ as a polynomial of one variable by fixing all of the variables other than $\lambda_j$ at $c$ (we specialize $\lambda_r = c_r$). There is only a finite number of monomials, and so there is only a finite number of corner roots between $h_i$ and any other monomial $h_n$ of $f$; we denote them as $x_k$. First, assume that $\{x_k\}$ is not an empty set. Tangibles are from an ordered monoid. Thus we can sort these corner roots along with the point $c$ by size. Considering the fact that $h_i$ is quasi-essential at $c$, we take the closest corner root in the array after which $h_i$ is inessential to be $x_s$. Since this corner root is the closest to $c$, $h_i$ and $h_s$ are quasi-essential at $x_s$, for otherwise there is another monomial $h_w$ bigger than $h_i$ but then $x_w$ is between $c$ and $x_s$ which is false. We obtained $x_s$ as a corner root on which $h_i$ is quasi-essential.\\

However, if the set of corner roots $\{x_k\}$ is empty we choose a variable other then $j$. If the corner roots sets are empty for all variables, it follows that the polynomial has no corner roots altogether, which is absurd. Indeed, if $f$ has any corner root, it must have at least one monomial which differs in at least one variable power from $h_i$. Leaving this variable free we obtain a non empty set of roots.
\end{proof}

We will prove that given $s(f(\layer{a}{x}))$ for all $x$, one can know all of the monomials that are quasi-essential at $a$. As a consequence, one can know all of the monomials of a polynomial~$f$, due to the lemma above that assures us that each monomial is at least quasi-essential at some corner root.\\

\begin{lem}
Let $f$ be a univariate polynomial, and let $g$ be the polynomial which contains all of the monomials of $f$ which are not inessential at some given point $a$. Then $g$ is of the form $g=c_a\sum_{k=0}^{n}{\layer{0}{b_k}} \lambda^ka^{n-k}$, where $c_a$ is tangible.
\end{lem}
\begin{proof}
Clearly, all of the monomials of $g$ have the same tangible value at $a$. Let $n$ be the degree of $g$, and define $c_a=t(g(a))a^{-n}$. Let $h$ be a monomial of $g$ of degree $k$. Then $c_aa^n=t(g(a))=t(h(a))=ba^k$. Therefore $t(b)= c_aa^{n-k}$, as desired.
\end{proof}

\begin{lem}
If $f$ and $g$ are two polynomials in one variable (with non-empty corner root sets), then $f$ and $g$ are root-equivalent and are equal at one point $\iff$ $f$ and $g$ have the same monomials $\iff$ $f=g$ everywhere.\\
\end{lem}
\begin{proof}

Fix a tangible corner root $a$ and look at the sum of all non-inessential monomials at~$a$; $$c_a\sum_{k=0}^{n}{\layer{0}{b_k}}\lambda^ka^{n-k}.$$ Next we look at the root at layer $i$: $f(\layer{a}{i})=c_a\sum_{k=0}^{n}{\layer{a^{n}}{b_k\cdot i^k}}$. Now its layer
$$s(f(\layer{a}{i}))=s(c_a\sum_{k=0}^{n}{\layer{a^{n}}{b_k\cdot i^k}})=\sum_{k=0}^{n} b_ki^k$$
(we include the layer of $c_a$ into the $b_k$'s, and therefore $s(c_a)=1$). This equation holds for all $i\in\mathbb{N}$. A similar argument for $g$ yields $s(g(\layer{a}{i}))=\sum d_ki^k$, with $e_a$ as the constant. Our goal is to show that $b_k=d_k$ and that $c_a$=$e_a$. This will prove that $f$ and $g$ have the same monomials, and the rest of the proof is trivial.\\

Next let us prove that $b_k=d_k$. We take a number $m\in\mathbb{N}$ so that $mb_k$ and $mc_k$ are in $\mathbb{N}$ for all $k$ (this is the $lcm$ of all of the denominators). Now take $i$ to be $\max_{k}(mb_k,md_k) + 1$. $\sum mb_ki^k$ is written to the base $i$, but this form is unique. Recalling $f(\layer{a}{i})=g(\layer{a}{i})$, we use the same argument for $g$ to obtain $mb_k=md_k$ and finally $b_k=d_k$.\\

We will now show that given the constant $c_a$ of a root $a$, all of the other constants are known. Let $\left\{a_i\right\}_{i=1}^{r}$ be the sorted set of roots of $f$ (and $g$). Fix a root $a_i$ and assume $c_{a_i}$ is known. Between two consecutive roots $a_k$ and $a_{k+1}$, there is one essential monomial (otherwise there would be another root between them). Clearly, this essential monomial is quasi-essential at the roots also. Therefore there is a monomial of the form $c_{a_{i}}\layer{0}{b_k}\lambda^ka_i^{n-k}$ and also of the form $c_{a_{i+1}}\layer{0}{b_k}\lambda^ka_{i+1}^{n-k}$. We conclude that $c_{a_{i}}a_i^{n-k}=c_{a_{i+1}}a_{i+1}^{n-k}$ and therefore we can calculate $c_{a_{i+1}}$ from $c_{a_i}, a_i$ and $a_{i+1}$. This argument works in the other direction as well, and thus all of $\left\{c_{a_i}\right\}_{i=1}^{r}$ are known.\\

Finally, since $f$ and $g$ are equal at some point, they have equal constants $c_a=e_a$ for some $a$. The rest of the constants are equal since they are calculated from this constant and the other roots (which are equal for $f$ and $g$). We obtained the desired result: if $f$ and $g$ are root-equivalent and are equal at some point, then they have exactly the same monomials, and therefore are equal everywhere. If they are equal everywhere, then clearly they are equal at some point and root-equivalent.
\end{proof}

We will now prove this monomial equality for several variables using an induction whose base is the above lemma. In several variables, the corner roots are much more complex and interesting as we will see later in this paper. First we prove a simple lemma which helps us show that any point of equality between two root-equivalent polynomials means equality at every point.

\begin{lem}
Let $f=\sum h_i $ be a polynomial in $n$ variables with a non-empty set of corner roots, and let $c=(c_1,...,c_n)$ be any point. Then there is a variable $\lambda_i$, and a corner root $a$ such that $a=(x_1,...,x_{i-1},c_i,x_{i+1},...x_n)$.
\end{lem}

\begin{proof}
Having a corner root set is equivalent to $f$ having at least two monomials, as we have seen when proving monomial equality. These two monomials must differ at the power of at least one variable, call it $\lambda_j$. Now fix any $\lambda_i=c_i$ such that $i\neq j$. Clearly we obtain a polynomial with $n-1$ variables and at least two different monomials. Therefore, this polynomial has a corner root $(x_1,...,x_{i-1},x_{i+1},...,x_n)$ and $a$ above is a corner root of $f$ as desired.
\end{proof}

\begin{thm}\label{moneq}
Let $f$ and $g$ be two polynomials with non-empty corner root sets, then $f$ and $g$ are root-equivalent and are equal at one point $\iff$ $f$ and $g$ have the same monomials $\iff$ $f=g$ everywhere.\\
\end{thm}
\begin{proof}
Assume that $f$ and $g$ are root-equivalent and there exists a point $c$ such that $f(c)=g(c)=x$. By the lemma above there is a variable, call it $\lambda_n$, and a corner root $a$ such that $a=(x_1,...,x_{n-1},c_n)$.

At any given corner root the layers behave like classical polynomials. For example consider $f(\lambda)=\layer{\lambda}{2}+\layer{0}{1}$ then $f(\layer{0}{x})=\layer{0}{2x+1}$. Since $f$ equals $g$ at every layer of the corner root, their essential monomials at the corner root must be equal up to multiplication by a constant. In view of the point of equality they are exactly equal.

Similarly to the proof of the above lemma, by choosing different variables and direction we prove that $f$ and $g$ must consist of exactly the same monomials, and the rest of the proof follows.
\end{proof}

\begin{cor}
If $f$ and $g$ are two polynomials with non-empty corner root sets such that $f$ and $g$ are root-equivalent, then there is some constant $c$ such that $f=cg$.\\
\end{cor}

\subsection{Primary polynomials}

\begin{lem}
Let $f$ be a polynomial in one variable. Then there is a factorization of $f$ into primary polynomials of its roots, and a power of $\lambda$.\\
\end{lem}

\begin{proof}
Take the highest $m\in\mathbb{N}$ such that $\lambda^m$ divides each monomial in $f$. After we factor out $\lambda^m$, the constant monomial is essential. We call the new polynomial $f_1$, and order the root set $\left\{a_i\right\}_{i=1}^{m}$ where $a_1<a_2<...<a_m$. The essential part at $a_1$ is primary in $a_1$. Let $j$ be the highest power of $\lambda$ in this primary polynomial, and let $c_j$ be the coefficient of $\lambda^j$. We normalize by multiplying by $c_j^{-1}$, and call this primary polynomial $g_1$.\\

Now take the essential part at $a_2$. The highest power of $\lambda$ of a monomial in $g_1$ is the smallest power here. Therefore by factoring out $\lambda^j$ where $j=\deg(g_1)$ of the essentials above, we get a primary polynomial $g_2$ in $a_2$. Let us look at $g_1g_2$ as a polynomial. First, it has the roots $a_1$ and $a_2$. Also, when $a<a_2$, $g_2(a)=b_n$. Therefore, $g_1g_2$ and $f_1$ have the same layers at images of all layers of $a_1$. When $a> a_1$, the essential monomial of $g_1$ at $a$ is $\lambda^j$ and therefore $f(a_2)=g_2(a_2)g_1(a_2)$. Thus $g_1g_2$ has the same layers as $f_1$ for $a_2$ as well.\\

We normalize $g_2g_1$ and continue to get $g_3$ as before. $g_3(a)$ for $a<a_3$ is exactly the constant we normalized $g_2g_1$ with, and therefore $g_3g_2g_1$ has the same layers on $a_1$, $a_2$ and $a_3$. We continue this process and obtain that $f_1=g_ng_{n-1}...g_1$ due to Theorem \ref{moneq}, and thus $f=\lambda^mg_ng_{n-1}...g_1$, as desired.
\end{proof}

\begin{exmp}
We will now factor a polynomial into primary polynomials. Consider $$f=4\lambda^5+5\lambda^4+6\lambda^3+6\lambda^2.$$
First we factor $\lambda^2$ and are left with
$$f_1=4\lambda^3+5\lambda^2+6\lambda+6.$$
The root set here is $\{a_1=0,a_2=1\}$. The essential monomials in $a_1$ are $6\lambda+6$. We normalize by $-6$ and obtain $g_1=\lambda+0$. The essential monomials in $a_2$ are $4\lambda^3+5\lambda^2+6\lambda$. We factor by $\lambda^j=\lambda^1$ and obtain $g_2=4\lambda^2+5\lambda+6$. $g_1(a)g_2(a)=g_1(a)\cdot 6=6a+6=f_1(a)$ for $a<a_2=1$, specifically $f_1(a_1)=g_1(a_1)g_2(a_1)$. When $a>a_1$ then $g_1(a)=a$; therefore $g_1(a)g_2(a)=a\cdot g_2(a)=f_1(a)$, specifically $f_1(a_2)=g_1(a_2)g_2(a_2)$. In this case, the factorization process ends here, and
$$f=\lambda^2 g_1 g_2=\lambda^2(\lambda+0)(4\lambda^2+5\lambda+6).$$
\end{exmp}

\subsection{Unique factorization}
\begin{thm}
Let $f$ be a univariate polynomial. Then $f$ factors uniquely into polynomials which are primary at the corner roots, possibly including $\lambda^m$ for some $m\in\mathbb{N}$ (up to multiplication by constants).\\
\end{thm}

\begin{proof}
By the above lemma we know that such a factorization exists. Suppose that there are two different factorizations of $f$, called $h_1$ and $h_2$. Take the smallest corner root $a$. Let $g$ be a primary polynomial in a root bigger than $a$. Then $c:=g(a)$ is the constant of $g$. Now $g(\layer{a}{i})=c$ for all $i$, since $a$ is smaller than the corner root of $g$. Thus, both $h_1$ and $h_2$ have the same layers at the primary in $a$ up to a constant  (since the other primaries yield constants in $a$). Thus the primaries in $a$ of both $h_1$ and $h_2$ are equal up to multiplication by a constant. We factor out the primaries in $a$, and continue this process to see that all the primaries are equal up to a multiplication by constants.
\end{proof}

Now the natural question is whether or not a primary polynomial factors uniquely into irreducibles. Define $\mathbb{P}_a$ to be the set of polynomials of the form $f=\sum_{k=0}^{n}\layer{0}{b_k}\lambda^ka^{n-k}$, and the function $\psi : \mathbb{P}_a \rightarrow \mathbb{Q}^+[\lambda]$ defined by $\psi(f)= \sum_{k=0}^{n}s(b_k)x^k$. $\mathbb{Q}^+$ is the set of positive rational numbers.\\

\begin{lem}
Let $f$ and $g$ be primary polynomials in $a$. Then $\psi(f+g)=\psi(f)+\psi(g)$ and $\psi(fg)=\psi(f)\psi(g)$. Moreover, $\psi$ is an isomorphism up to multiplication by a tangible constant in $\mathbb{P}_a$.\\
\end{lem}
\begin{proof}
Let $f$ be of the form $f=\sum_{i=0}^{n}a_i\lambda^i$ and $g=\sum_{j=0}^{m}a_j\lambda^j$. Assume $n\geq m$; then $$f+g=\sum_{k=0}^{m}(a_k+b_k)\lambda^k + \sum_{k=m+1}^{n}(a_k)\lambda^k.$$ Therefore, $$\psi(f+g)=\sum_{k=0}^{m}(s(a_k+b_k))x^k+ \sum_{k=m+1}^{n}(a_k)\lambda^k=\sum_{k=0}^{m}(s(a_k)+s(b_k))x^k+ \sum_{k=m+1}^{n}(s(a_k))\lambda^k=$$$$=\sum_{k=0}^{n}s(a_k)x^k + \sum_{k=0}^{m}s(b_k)x^k=\psi(f)+\psi(g).$$\\

Now for multiplication, $$fg=\sum_{i,j=1}^{i=n,j=m}(a_i+b_j)\lambda^{i+j}.$$ $$\psi(fg)=\sum_{i,j=1}^{i=n,j=m}(s(a_i+b_j))x^{i+j}=\sum_{i,j=1}^{i=n,j=m}(s(a_i)+s(b_j))x^{i+j}=\psi(f)\psi(g).$$

Next, $\psi$ is clearly onto, and in order to prove it is also an isomorphism we must verify it is injective. Assume $\psi(f)=\psi(g)$, $\sum_{k=0}^{n}s(b_k)x^k=\sum_{k=0}^{n}s(a_k)x^k$. We see that $\forall k, s(a_k)=s(b_k)$. Since both $f$ and $g$ are primary in $a$, the tangible values of $a_k,b_k$ are fixed up to a multiplication by a tangible constant $c$, and so we get $f=cg$.
\end{proof}

\begin{cor}
We have unique factorization for primary polynomials iff there is unique factorization in $\mathbb{Q}^+[\lambda]$.\\
\end{cor}

Unfortunately, unique factorization fails in polynomials over the positive rational numbers (or even the positive real numbers). Take the following example: $$x^6+2x^5+3x^4+2x^3+3x^2+2x+2=(x^4+x^2+1)(x^2+2x+2)=(x^2+x+1)(x^4+x^3+x^2+2).$$
In order to verify that these polynomials are irreducible, we will look at the complete factorization over~$\mathbb{R}$:
$$x^6+2x^5+3x^4+2x^3+3x^2+2x+2=(x^2+x+1)(x^2-x+1)(x^2+2x+2).$$
One can easily see that the above factorizations are obtained by combining the left pair of polynomials [$(x^2+x+1)(x^2-x+1)$] or the right one [$(x^2-x+1)(x^2+2x+2)$]. It is immediate that primary polynomials do not factor uniquely in this structure. (I thank Professor Uzi Vishne for his elegant counterexample.)\\

\section{Full Rational Layer}

We have seen that unique factorization fails for primary polynomials, due to the lack of negative layers. Unlike positive layers, negative layers first arise in light of the factorization problem. We will see that this expansion will be interesting in itself, and will solve our factorization problem.\\

From now on, we fix $L=\mathbb{Q}$.\\

\subsection{Layer zero elements}
One should introduce a new layer -- zero, which changes the rules. Consider the primary polynomial $\lambda^2 + 2\lambda + 4$; here $2\lambda$ is quasi-essential. However, if we change the polynomial slightly into $\lambda^2 + \layer{2}{0}\lambda + 4$, we turn the middle monomial into an inessential one. This is true since this monomial does not change the size of the polynomial at any point, but it also does not change the layer because it contributes zero. These layer zero monomials will be the only exception to unique factorization in one variable.\\

For example, let us look at the following polynomial:
$$f=\layer{0}{0}(\lambda^2+1\lambda+0).$$
$f$ is equal to $\layer{0}{0}(\lambda+\layer{1}{y})(\lambda+\layer{(-1)}{x})$ for any $x,y\in\mathbb{Q}$. Unfortunately, this is not the only type of counterexample to unique factorization. However, the only counterexamples to unique factorization involve layer zero monomials which do not factor polynomials without zero layer monomials, as we will see in the next subsections.\\

\subsection{Monomial equality}
First we must obtain the same monomial equality result that we have seen only for positive layers. We wish to prove that if two polynomials are equal as functions, they have the same essential and quasi-essential monomials.\\

We notice that root-equivalence is weaker than monomial equivalence in this structure. Consider the polynomials $f=2\lambda^4+\layer{3}{0}\lambda^3+0$, $g=\lambda^4+\layer{2}{0}\lambda^2+0$. Both $f$ and $g$ have corner roots $1$ and $-1$. Also $\forall i,f(\layer{(-1)}{i})=g(\layer{(-1)}{i})=0$ and $s(f(\layer{1}{i}))=s(g(\layer{1}{i}))=i^4$. Therefore $f$ and $g$ are root-equivalent, but clearly not equal. We need to capture the monomials with layer zero coefficients which are lost on corner roots because they are inessential there. We capture these monomials by looking at the slopes of the polynomial as well.\\

\begin{lem}
Let $p_1$ and $p_2$ be two primary polynomials in $a$. Assume that $\forall \ell \in\mathbb{Q},s(p_1(\layer{a}{\ell}))=s(p_2(\layer{a}{\ell}))$, and assume that $p_1$ and $p_2$ have no monomials with layer zero coefficients. Then $p_1$=$cp_2$ for some constant $c$.\\
\end{lem}
\begin{proof}
 Assume that $\deg(p_1)\geq\deg(p_2)$. Now choose a polynomial $p$ with high layer coefficients so that $p_1+p$ and $cp_2+p$ are both primary polynomials in $a$ with positive layer coefficients ($c$ is a constant). Also we choose $p$ so that it does not have any monomials in which both $p_1$ and $cp_2$ are missing. Due to our results in the positive layer section, we know that $p_1+p=cp_2+p$. Now since $p_1$ and $p_2$ do not have layer zero monomials, we can remove $p$ from the equations by adding $\layer{0}{-1}p$: $p_1 + p +  \layer{0}{-1}p$ is $p_1$ with the addition of layer zero monomials; assume $h$ is such a monomial. Since $p_1$ does not have monomials of layer zero $h$ cannot be a monomial of $p_1$, and therefore it must be a monomial of $p$. $cp_2 + p +  \layer{0}{-1}p=p_1 + p +  \layer{0}{-1}p$ so $h$ must not be a monomial of either $p_1$ nor $p_2$ implying $h$ is in $p$ which is absurd. We conclude that $p_1 + p +  \layer{0}{-1}p$ has no layer zero monomials and is equal to $p_1$, and $cp_2 + p +  \layer{0}{-1}p=cp_2$ so $p_1=cp_2$.
\end{proof}

\begin{lem}
 Let $x,y$ be two tangibles, and $r,t \in R$. Then there is no more than one monomial $a\lambda^k$ such that $ax^k=r$ and $ay^k=t$.
\end{lem}
\begin{proof}
 This is a case of two linear equations in two variables. $a = (x^{-1})^kr$ since $ax^k=r$ and therefore $(x^{-1})^kry^k=t$. Due to our assumptions, there is only one $k$ such that  $(x^{-1}y)^k=r^{-1}t$. Then $a=r(x^{-1})^k$.
\end{proof}

Since the layer zero is not our main interest and it turns out to be a counterexample to unique factorization, we will prove an easier theorem for monomial equality. Instead of building an exact description of the minimal amount of data we need in order to reconstruct the polynomial, we assume that we have all the data.\\

\begin{thm}\label{moneqmult}
 For any two polynomials $f$ and $g$, $f$ and $g$ have the same monomials if and only if $f=g$ everywhere.\\
\end{thm}
\begin{proof}
Assume $f=g$ everywhere. We already know that each monomial with layer different from zero must be quasi-essential at some point. As we have seen, this monomial has to be in the polynomial (both in $f$ and $g$).\\

Assume that $h$ is a monomial with layer zero. Then $h$ must be essential at some point $a=(a_1,...a_n)$. As we have seen, when all variables are specialized to $a$ except for $\lambda_i$, then $h$ is a monomial in one variable which is essential at least between $x<a_i<y$ for some $x,y$. Due to our assumptions there are at least two points $x',y'$ such that $x<x'<y'<y$ and the monomial is essential at $x'$ and $y'$. Thus $h(x')=f(x')$ and $h(y')=f(y')$, and as a consequence of the above lemma, $h$ is known. Therefore we know the exponent of $\lambda_i$. We can calculate the exponent of all the variables this way, and then the coefficient is given because we know $h(a)=f(a)$.\\

In conclusion, any monomial of $f$ or $g$ is uniquely determined by the polynomial as a function, as desired. The other direction is trivial.
\end{proof}

\subsection{Regular polynomials}
We will describe the multiplication process of polynomials, since it differs from classic multiplication and is essential to our paper. Even though multiplication is distributive, some monomials of the product are inessential and therefore deleted, since we view polynomial as functions. The multiplication is well defined due to theorem \ref{moneqmult}.\\

\begin{exmp}
Consider the two polynomials $ f=(-2)\lambda^2+\lambda+1$, $g=(-5)\lambda^3 + (-1)\lambda + 0$. Both $f$ and $g$ have corner roots $1$ and $2$. Thus $f\cdot g =$
$$=((-2)\lambda^2)((-5)\lambda^3- 1\lambda) + \lambda((-5)\lambda^3 + (-1)\lambda + 0) + 1((-1)\lambda + 0)=$$
$$=(-7)\lambda^5 + (-5)\lambda^4+(-3)\lambda^3+(-1)\lambda^2 + \layer{0}{2}\lambda + 1.$$
Whereas $(-7)\lambda^5$ is essential at $\lambda > 2$, $(-7)\lambda^5 + (-5)\lambda^4+(-3)\lambda^3+(-1)\lambda^2$ are quasi-essential at $\lambda=2$, $(-1)\lambda^2$ is essential at $1<\lambda<2$, $(-1)\lambda^2 + \layer{0}{2}\lambda + 1$ are quasi-essential at $\lambda =1$, and $1$ is essential at $\lambda <1$.
\end{exmp}

\begin{defn}
A polynomial $f=\sum \layer{a_I}{\ell_I}\lambda^I\in R[\lambda_1,...,\lambda_n]$ is called \textit{regular} if $\ell_I\neq 0$ for all $I$.
\end{defn}

We wish to prove that regular polynomials have unique factorization. In order to give this result more meaning, we would also like to prove that the product of two regular polynomials is also regular.\\

\begin{lem}
Let $g$ and $h$ be polynomials. A monomial of $f=gh$ is essential at $a$ $\iff$ it is a product of an essential monomial of $g$ at $a$ and an essential monomial of $h$ at $a$.\\
\end{lem}
\begin{proof}
A monomial of $f$ is the sum of products of monomials from $g$ and monomials from~$h$. For example, $g=\lambda+0$, $h=\lambda+\layer{0}{2}$, $f=gh=\lambda^2+\lambda^{(3)}+\layer{0}{2}$. The middle monomial $\lambda^{(3)}$ is the sum of $\lambda\cdot \layer{0}{2}+0\cdot\lambda$.\\

Assume that $u$ is an essential monomial of $f$ at point $a$, then $u=g_1h_1+...+g_kh_k$. If for any $i$ $g_i$ of $h_i$ are inessential then $u$ is inessential, thus at $a$ for all of the monomials $g_i,h_i$ are quasi-essential. Therefore $g_ih_j$ is also quasi-essential at $a$ for every $1\leq i,j \leq k$. Assume $k>1$, the monomials $g_1h_1$ and $g_1h_2$ must have different powers of the variables otherwise $h_1$ and $h_2$ have the same powers which is absurd. We obtained two different monomials which are quasi-essential at $a$ which contradicts the assumption that $u$ is essential at $a$, and thus $k=1$.\\

As a consequence of the contradiction, any essential monomial of $f$ is a product of an essential monomial of $g$ and an essential monomial of $h$ as desired.\\

In the other direction, assume $u$ is a monomial of $f$ which is the product of monomials of $g$ and $h$ which are essential at $a$, $u=g_1h_1$. Clearly $u$ is essential at $a$ as any other product of two monomials $g_ih_j$ is inessential at $a$ due to $g_i$, $h_j$ or both.
\end{proof}

\begin{cor}\label{regfactor}
Let $f=gh$ be a factorization of a polynomial $f$. Then $f$ is regular $\iff$ $g$ and $h$ are regular.\\
\end{cor}
\begin{proof}
Any monomial of layer zero which is not essential at a certain point must be inessential at that point. Thus in order for a polynomial to be regular, all of its essential monomials must be of layer different from zero.\\

A product of monomials $f_i=g_jh_k$ is an essential monomial of $f$ at $a$ $\iff$ $g_j$ and $h_k$ are essential at $a$ due to the corollary above. The layer of $f_i$ is the product of the layers of $g_j$ and $h_k$, Thus the layer of $f_i$ is zero $\iff$ the layer of either $g_j$ or $h_k$ is zero.\\

Therefore $f$ is regular $\iff$ all monomials of $f$ are of layer other than zero $\iff$ all of the monomials of $g$ and $h$ are with layer different than zero $\iff$ $g$ and $h$ are regular.
\end{proof}

\subsection{Factorization}

\begin{defn}
Given an essential monomial $h$ of a polynomial $f$, the largest corner root of $f$ at which $h$ is quasi-essential is called the \textit{big root} of $h$ and the smallest corner root is called the \textit{small root} of $h$. Note that $h$ may have one corner root which is both the big root and small root of $h$.\\
\end{defn}

\begin{exmp}
Consider polynomial
$$f=\lambda^2 + \lambda + \layer{0}{1/4}.$$
The only corner root of $f$ is$\layer{0}{-1/2}$. Thus it is both the big root and small root of every monomial.\\
\end{exmp}

\begin{lem}
Let $f$ be a polynomial with a constant term, and let $h$ be an essential monomial of $f$ so that $s(h)\neq 0$, and the big root and small root of $h$ are distinct. Then $f$ factors uniquely, up to multiplication by a constant into $f=g_1g_2$, so that $g_1$ contains all of the corner roots greater or equal to the big root of $h$, and $g_2$ contains all of the corner roots less than or equal to its small root. Also any irreducible factor of $f$ divides either $g_1$ or $g_2$.\\
\end{lem}

\begin{proof}
Let $f=\sum_{i=1}^{n} c_i\lambda^i$ and $h=c_k\lambda^k$ for some $k$. Define $g_2=\sum_{i=1}^{k}c_i\lambda^i$, and $g_1=c_k^{-1}\sum_{i=k}^{n}c_i\lambda^{i-k}$ (here we use the fact that $s(h)\neq 0$). Multiplication is very easy since the corner roots of $g_1$ are bigger than those of $g_2$. When a point has value smaller or equal to the small root of $h$, then in $g_1$ the essential monomial is $0$. Therefore $0\cdot g_2$ are essential and quasi-essential monomials of $g_1g_2$. When a point has value greater than the small root of $h$, the only essential monomial of $g_2$ is $c_k\lambda^k$. Multiplying by $g_1$ we get $c_k\lambda^kg_1=c_k\lambda^kc_k^{-1}\sum_{i=k}^{n}c_i\lambda^{i-k} = \sum_{i=k}^{n}c_i\lambda^{i}$. Together we get $g_1g_2=f$.\\

Now, assume that there is another factorization $f=u_1u_2$ with the same properties. Since $f$ has a constant term, then so must $u_1$ and $u_2$. The constant of $u_1$ multiplied by $u_2$ must equal monomials of $f$ which are essential and quasi-essential at the corner roots of $u_2$ and $g_2$, as before. So $g_2=cu_2$ for some constant $c$. By a similar argument, $d\lambda^ku_1$ are the remaining monomials of $f$ for some constant $d$, and so the factorization $f=g_1g_2$ is unique up to multiplication by constants.\\

Let $f=J_1J_2...J_l$ be a factorization of $f$ into irreducible factors. We will show that each $J_i$ has corner roots which are all larger than the big root of~$h$ or all smaller than the small one. Otherwise, there is a factor $J=J_i$ that has a corner root bigger than the big root of $h$ and a corner root smaller than the small root of $h$. The polynomial $J$ is irreducible, and so from the first part of the theorem we have already proved, the essential monomials which are not at the edges must have layer zero coefficient (or $J$ would factor further). Therefore between the big and small roots of $h$, the essential monomial of $J$ is of layer zero, and so the essential monomial at this point of $f$ is of layer zero which is absurd.
\end{proof}

\subsection{Primary polynomials}
We have expanded the layer structure further in order to solve the problem of unique factorization of primary polynomials. Next we will verify that this property indeed holds.\\

Recall that the function $\psi : \mathbb{P}_a \rightarrow \mathbb{Q}[x]$, where $\mathbb{P}_a$ is the semiring of primary polynomials in $a$ and
$$\psi(f)= \sum_{k=0}^{n}s(b_k)x^k$$
for $f=\sum_{k=0}^nb_k\lambda^k$.
We will use $\psi$ here to prove unique factorization for primary polynomials.\\

In the case of positive layers, $\psi$ is an isomorphism. In this structure, we have for example $\psi(\layer{\lambda}{0} + 0) = 1 = \psi(\layer{\lambda^2}{0} + 0)$, so $\psi$ is not an isomorphism. Fortunately, it is still a homomorphism, since this part of the proof is not affected by the existence of layer zero.\\

\begin{lem}
Let $f$ and $g$ be primary univariate polynomials without monomials of layer zero. Then $\psi(f)=\psi(g) \iff f=cg$ for some tangible constant $c$ .
\end{lem}
\begin{proof}
The proof is similar to the positive layer case. Knowing that $f$ and $g$ has no layer zero monomials leaves only one choice up to a tangible constant for $\psi^{-1}(\psi(f))$.
\end{proof}

It is important to note that primary polynomials cannot have layer zero monomials other than the leading monomial and the constant. The reason is that a quasi-essential monomial with layer zero is an inessential monomial and therefore is not part of the polynomial.\\

\begin{thm}
Let $f$ be a primary univariate polynomial (with $L=\mathbb{Q}$). Then $f$ factors uniquely into irreducible factors (up to multiplication by a constant).\\
\end{thm}
\begin{proof}
Assume that $f$ is a regular primary polynomial, and that it factors in two different ways $f=g_1\cdots g_n=h_1\cdots h_m$. Since $f$ is regular, so are $g_1,...,g_n,h_1,...,h_m$. As a consequence of the lemma, these factorizations are the same up to a multiplication by a constant.\\

Now we take $f$ with a leading monomial of layer zero. $f$ must be of the form $f=\layer{\lambda^n}{0} + g$ for $g$ with $\deg(f)> \deg(g)$ and $g$ with lead monomial of layer different than zero. First we will show that $f$ factors into
$$(\layer{\lambda}{0}+a)^{n-\deg(g)}(-a)^{n-\deg(g)}g=(\layer{\lambda^{n-\deg(g)}}{0}+a^{n-\deg(g)})(-a)^{n-\deg(g)}g=\layer{\lambda^n}{0} + g = f.$$
Note that $g$ has a leading monomial coefficient of size $a^{n-\deg(g)}$, and note that we strike out any middle monomial with layer zero.\\

Take a factorization $f=g_1\cdots g_k$. At least one factor must have a layer zero leading monomial. We rearrange and rename the factors to obtain $f=gh$, where $g$ has leading monomial of layer zero and $h$ does not. The polynomial $f$ factors further to $(\layer{\lambda}{0}+a)^{l}(-a)^{l}uh$ where u has no leading monomial of layer zero, and neither does $uh$. Since the sum of the lowest degree monomials of $f$ equals to $uh$ and also to $g$ we have $uh=g$, and $l=n-\deg(g)$. Thus any factorization of $f$ is a sub-factorization of $(\layer{\lambda}{0}+a))^{n-\deg(g)}$.\\

A similar argument applies to $f$ with layer zero constant. $f=g\lambda^k+\layer{a^{\deg(g)+k}}{0}$ where $g$ has a constant term $a^{\deg(g)-k}$. We factor $f$; $f=g(\lambda+\layer{a}{0})^k$. Like the leading monomial of layer zero argument, $f$ factors uniquely into $g\lambda^k+\layer{a^{\deg(g)+k}}{0}$.\\

In conclusion, a general primary polynomial $f$ factors into $f=(\layer{\lambda}{0}+a)^m(\lambda+\layer{a}{0})^kg$ where $g$ has no layer zero monomials, and not a layer zero constant. It is easy to see that the middle monomials of this product are $a^m\lambda^kg$, therefore $g,k$ and $m$ are determined uniquely according to the middle monomials of $f$. We already know that such $g$ factors uniquely, and so every primary polynomial~$f$ factors uniquely into irreducible factors.
\end{proof}

\subsection{Main result}
We are now ready to prove our main theorem.

\begin{thm}
Any regular polynomial $f$ factors uniquely into irreducible factors.\\
\end{thm}
\begin{proof}
From previous sections we know that $f$ factors uniquely into primary polynomials around each essential monomial with layer different from zero. Thus a regular polynomial $f$ factors uniquely into primary polynomials at each of its corner roots. Moreover, we proved that primary polynomials factor uniquely into irreducible factors. Thus $f$ factors uniquely into irreducible factors.
\end{proof}

\subsection{Non-regular polynomial factorization}
Next we will discuss non-regular polynomials. We will first describe the basic irreducible factors, then examine some counterexamples to unique factorization, and thereafter we will sum up the factorization process for a general polynomial.\\

\subsubsection{Basic irreducible factors}
Unlike the case of positive layers, here we have an irreducible polynomial that is not primary. Fortunately, this is the only form of an irreducible polynomial other than the primary polynomials
\begin{equation}\label{basici}
\lambda^m + \layer{0}{0}b\lambda^k + c.
\end{equation}
We list three properties of this form:
\begin{enumerate}
\item $c$ must be of layer different from zero, or this polynomial will be reducible.
\item Not all polynomials of this form are irreducible.
\item We choose this polynomial not to be primary; thus it has two corner roots.
\end{enumerate}

\begin{lem}
Non primary polynomials of the form
$$\lambda^m +\layer{b}{0}\lambda+c$$
and
$$\lambda^{k+1} + \layer{b}{0}\lambda^k+c,$$
are irreducible.
\end{lem}
\begin{proof}
Since a polynomial of this form has two corner roots, there are three types of polynomials that can factor it: Primaries in the first corner root, primaries in the second corner root, and polynomials with both of these roots. We will try all of these options to determine when these polynomials are irreducible.\\

First, we will try to multiply two primary polynomials, one for each corner root. Since the leading monomial and the constant do not have layer zero, the leading monomial and constant of their factors also must not be of layer zero. Therefore any such primaries are regular (since quasi-essential monomials of ghost layer zero are deleted). Clearly, this product cannot yield a polynomial of the form \ref{basici} (since the product is regular as well).\\

Next, we try to multiply a primary polynomial in the small root, with a polynomial having two corner roots. It is easier to look first at the essential monomials only, so we will assume $g$ and $h$ have no monomials which are never essential. Define $g=r+w$, $h=t+u+v$ where $r,w,t,u,v$ are monomials. Then $gh=tr+ur+uw+rv+vw$ where $uw$ and $rv$ are never essential. In order for the result to be of the desired form, $s(uw)$ must be zero, and $s(r),s(w),s(t),s(v)$ must not be zero; therefore $s(u)=0$. We remain with the monomial $rv$ which is never essential, and therefore we have failed to achieve the desired form.\\

We get a similar result by multiplying a primary in the big root together with a polynomial with both corner roots. $g=r+w$, $h=t+u+v$, $gh=tr+tw+ru+uw+vw$ with $tw$ and $ru$ never essential. We must have $s(uw)=0,s(w)\neq 0 \Rightarrow s(u)=0$, and $s(r),s(t),s(v)\neq 0$. We again get a polynomial with a monomial which is never essential.\\

Now we multiply two polynomials both having two corner roots. $g=r+y+t$, $h=u+v+w$, $gh=ru+rv+uy+yv+bw+tv+tw$ with the essentials $ru,yv,tw$. We do not want any monomials which are never essential, so we have $s(rv),s(uy),s(yw),s(tv)=0$. Clearly $s(r),s(u)\neq 0$, and so $s(y),s(v)=0$. Finally we obtain $gh=ru+yv+tw$. This is the only way a polynomial of the above form will factor, and when this factorization is impossible then the polynomial is irreducible.\\

Since our $g$ and $h$ each have three monomials, then $y$ and $v$ must have degree at least $1$ in $\lambda$. Thus $f=gh=ru+yv+tw$ must have $\deg(yv)\geq 2$, and any polynomial of the form
$$\lambda^m +\layer{b}{0}\lambda+c$$
which is not primary is irreducible. Moreover, since $\deg(r) > \deg(y)$ and $\deg(u) > \deg(v)$, then $\deg(ru) \geq \deg(yv)+2$. Thus the second type of irreducible non-regular polynomials is
$$\lambda^{k+1} + \layer{b}{0}\lambda^k+c.$$
\end{proof}
We will later see that these are the only examples of irreducible polynomials which are not primary.\\

\subsubsection{Counterexamples}
We provide a few counterexamples to unique factorization for non-regular polynomials. We will start with the trivial case, and finish with more complex cases.\\

We have already seen that a product of irreducible non-regular polynomials is a product of pairs of matching monomials. Changing the layers of matching monomials, we can produce different factorizations:
$$(\lambda^2 + \layer{1}{0}\lambda + \layer{0}{1/3})(\lambda^2 + \layer{1}{0}\lambda + \layer{0}{3})=(\lambda^4 + \layer{2}{0}\lambda^2 + 0)= (\lambda^2 + \layer{1}{0}\lambda + 0)^2.$$

Next, we notice that layer zero coefficients eliminate quasi-essential monomials. We will build an example with different quasi-essential monomials which disappear. We will multiply an irreducible polynomial, with a primary polynomial at a corner root which is between the big and small root of the irreducible one. We will first study the general case, and then bring a concrete example.\\

Take $g=a+\layer{b}{0}+c$ with big root $a_1$ and small root $a_2$, and take $h=d+e+f$ primary at $a_3$ so that $a_1>a_3>a_2$. Think of $e$ as any number of quasi-essential monomials (including none). Then $f=gh=ad + \layer{0}{0}(bd+be+bf) + cf$ where $be$ is quasi-essential of layer zero and so is inessential.\\

A concrete example: $g=(\lambda^2 + \layer{1}{0}\lambda + 0)$ with corner roots $-1,1$, and a primary at $0$: $h=\lambda^2+\lambda+0$. Then $f=gh=\lambda^4+\layer{1}{0}\lambda^3 + \layer{0}{0}1\lambda+0$. The monomial $1\lambda^2$ would be quasi-essential at $0$, but the layer zero makes it inessential. If instead $h=\lambda^2+0$, the product $f=gh$ remains the same.\\

We have seen how to factor a non-regular polynomial with two corner roots. We will now give an example of different ways to factor the same polynomial.\\
\begin{exmp}
Take $f=(-2)\lambda^6 + \layer{\lambda^4}{0}+0$ with corner roots $1$ and $0$. Then
$$f=((-1)\lambda^3 + \layer{\lambda^2}{0}+0)^2$$ but also
$$f=((-1)\lambda^2 + \layer{\lambda}{0}+0)((-1)\lambda^4 + \layer{\lambda^3}{0}+0).$$
\end{exmp}

So far we have seen various examples of non-unique factorizations of polynomials with the same corner roots. Now we will show an example of factorization into polynomials with different corner roots.
\begin{exmp}
$$f=1\lambda^8 + \layer{4}{0}\lambda^5+\layer{4}{0}\lambda^4 + 0.$$
This polynomial has corner roots $\{-1,0,1\}$. We factor it into two factors, the first with corner roots $\{-1,1\}$ and the second with corner roots either $\{-1,0\}$ or $\{0,1\}$.
$$f=(\lambda^6+\layer{3}{0}\lambda^3+0)(1\lambda^2+\layer{1}{0}\lambda+0)$$
$$f=(2\lambda^6+\layer{4}{0}\lambda^4+0)((-1)\lambda^2+\layer{\lambda}{0}+0).$$
These polynomial are not all irreducible, but they factor into polynomials having the same corner roots, which differ in the two cases.\\
\end{exmp}

\subsection{Irreducible polynomials}
We have introduced the basic irreducible polynomial, but we have yet to prove that it is the only non-regular irreducible polynomial. We will answer this question now.

\begin{thm}
The only irreducible non-primary polynomials are the basic irreducible polynomials.
\end{thm}
\begin{proof}
As we have seen before, any polynomial with a middle essential monomial with layer different from zero is reducible. Therefore we are only interested in polynomials which are not primary, having middle essential monomials of layer zero.\\

We start with the case of a polynomial with more than one corner root (since it is not primary) with quasi-essential monomials of layer different from zero at some corner root $a$. The polynomial is of the form $f=a_n\lambda^n+...+\layer{a_k}{0}\lambda^k + a_m\lambda^m +...+a_0$, where $\layer{0}{0}a_k\lambda^k + a_m\lambda^m$ have the same tangible value at $a$. We claim that
$$f=(a_m\lambda^m +...+a_0)a_m^{-1}(a_n\lambda^{n-m}+...+\layer{a_k}{0}\lambda^{k-m}+a_m).$$
It is easy to see that the corner roots of the left polynomial are the corner roots of $f$ which are less than or equal to $a$, and the right polynomial has all of the corner roots of $f$ which are larger or equal to $a$. When $\lambda > a$, the essential part is
$$(a_m\lambda^m)a_m^{-1}(a_n\lambda^{n-m}+...+\layer{a_k}{0}\lambda^{k-m})=a_n\lambda^n+...+\layer{a_k}{0}\lambda^k.$$
When $\lambda \leq a$ the essential part is
$$(a_m\lambda^m +...+a_0)a_m^{-1}(a_m)=a_m\lambda^m +...+a_0.$$
(Note that the quasi-essential monomials at $a$ multiplied by layer zero are inessential.) Thus we have proved that $f$ factors into the above polynomials.\\

We are left with the case of polynomials with three or more corner roots and only essential monomials:
$$f=g+b_1\lambda^{m_1}+b_2\lambda^{m_2}+b_3\lambda^{m_3}+b_4,$$
where $g=\lambda^{m_1+1}g'$ for some polynomial $g'$. Let $a_n>a_{n-1}>...>a_1$ be the corner roots of $f$, $n \geq 3$. We claim that
$$f=(b_2^{-1}b_3\lambda^{m_3-m_2}(g+b_1\lambda^{m_1})+b_3\lambda^{m_3}+b_4)(b_2b_3^{-1}\lambda^{m_2-m_3}+0).$$
The corner root of the second polynomial is $a_2$. Indeed, if $t(b_2x^{m_2})=t(b_3x^{m_3})$ then $t(b_2b_3^{-1}x^{m_2-m_3})=0$. It is easy to verify that the corner roots of the right polynomial are $a_n,a_{n-1},...,a_3,a_1$. The only non-trivial case is of $a_3$: $$b_2^{-1}b_3\lambda^{m_3-m_2}b_1\lambda^{m_1}+b_3\lambda^{m_3}=b_3\lambda^{m_3}(b_2^{-1}b_1\lambda^{m_1-m_2}+0).$$
Therefore the product indeed equals $f$.

Note that this holds for polynomials with leading monomial and/or constant of layer zero.
\end{proof}

\subsection{Summary of the factorization process}
In this section we describe the factorization process for a general polynomial, and then give an example.\\

The first step is partitioning the polynomial around its essential middle monomials of layer different from zero. Afterwards, we factor the polynomial until we get to an irreducible factor, or to a primary polynomial. We factor the primary polynomials so they do not have layer zero leading monomials or constants. Finally, we factor the regular primary part in the same way polynomials factor over $\mathbb{Q}$.\\

\subsubsection{Example}
Consider the polynomial:
$$f=\layer{(-10)}{0}\lambda^{10}+(-4)\lambda^8 + (-1)\lambda^7+\layer{3}{0}\lambda^5+ 5\lambda^3+\layer{5}{-1}.$$

First we see that $\lambda=0$ is the smallest corner root since it is the solution of the equation $5\lambda^3+\layer{5}{-1}$. $5\lambda^3$ is the essential monomial with two distinct big and small roots, and thus $f$ factors into
$$f=(\layer{(-15)}{0}\lambda^{7}+(-9)\lambda^5 + (-6)\lambda^4+\layer{(-2)}{0}\lambda^2+0)(5\lambda^3+\layer{5}{-1})=$$
$$=(-10)(\layer{\lambda^{7}}{0}+6\lambda^5 + 9\lambda^4+\layer{13}{0}\lambda^2+15)(\lambda^3+\layer{0}{-1}).$$

Now the smallest corner root of $\layer{\lambda^{7}}{0}+6\lambda^5 + 9\lambda^4+\layer{13}{0}\lambda^2+15$ is $\lambda=1$. The next essential monomial is $\layer{13}{0}\lambda^2$, which has layer zero and therefore we will factor it later to basic irreducible factors. The next corner root is $\lambda=2$, and the essential monomial is $9\lambda^4$. Therefore $f$ factors further into
$$f=(-10)(\layer{(-9)}{-1}\lambda^{3}+(-3)\lambda + 0)(9\lambda^4+\layer{13}{0}\lambda^2+15)(\lambda^3+\layer{0}{-1})=$$
$$=(-10)(\layer{\lambda^{3}}{0}+6\lambda + 9)(\lambda^4+\layer{4}{0}\lambda^2+6)(\lambda^3+\layer{0}{-1}).$$

We obtained three factors -- primary at $3$, the basic irreducible factor at $2$ and $1$, and a primary at $0$. We further factor each of these polynomials: $$\layer{\lambda^{3}}{0}+6\lambda + 9=(\layer{\lambda}{0}+3)^2(\lambda + 3)$$
$$\lambda^4+\layer{4}{0}\lambda^2+6=(\lambda^2+\layer{2}{0}\lambda+3)^2$$
$$\lambda^3+\layer{0}{-1} = (\lambda + \layer{0}{-1})(\lambda^2+\lambda+0).$$
Thus:
$$f=(-10)(\layer{\lambda}{0}+3)^2(\lambda + 3)(\lambda^2+\layer{2}{0}\lambda+3)^2(\lambda + \layer{0}{-1})(\lambda^2+\lambda+0).$$

\section{Several Variables}

\subsection{Unique factorization of primary polynomials}

Since changing the layer of the primary point $a$ does not change its being a primary point, we will assume from now on that the primary point is tangible.\\

\begin{thm}
Let $f$ be a regular primary polynomial at $a$. Then $f$ factors uniquely into irreducible factors.\\
\end{thm}
\begin{proof}
The key to this proof is to build a homomorphism between primary polynomials at a given point and polynomials over $\mathbb{Q}$. First we need to verify that a primary polynomial factors into primary polynomials. Let $g$ and $h$ be polynomials so that $f=gh$. Assume $g$ has a monomial $u$ which is inessential at $a$, and let $v$ be any monomial of $h$. Then $uv$ is inessential at $a$. However, $u$ must not be inessential so $us$ must be at least quasi-essential for some $s$ in $h$. Thus $us$ is a monomial of $f$ which is not quasi-essential at $a$, which contradicts the definition of a primary polynomial.\\

We define $\mathbb{P}_a$ to be the set of regular polynomials $f$ which are primary in $a$ and for which the tangible value of $f(a)$ is 0. The latter means that the constant $c$ in the form $f=c\sum \layer{0}{b_I}a_I\lambda^I$ is $0$. We will show that $\mathbb{P}_a$ is closed under multiplication, and also under factorization.\\

Assume $g,h \in \mathbb{P}_a$. Clearly, any monomial $u$ of $f=gh$ is the sum of products $u=g_1h_1+...+g_kh_k$ where $g_i$ are monomials of $g$, and $h_i$ are monomials of $h$. Since $g$ and $h$ are primary in $a$ and their tangible value is $0$, the tangible value of $u(a)$ is also $0$. Due to Corollary \ref{regfactor}, $f$ is regular. Thus $f$ is primary in $a$ and the tangible value of $f(a)$ is $0$, and in other words $f \in \mathbb{P}_a$.\\

Assume $f \in \mathbb{P}_a$, and assume $f$ factor into $f=gh$. Due to Corollary \ref{regfactor}, $g$ and $h$ are regular.  Assume $g$ has a monomial $u$ which is inessential at $a$. $u$ must not be completely inessential, so $u$ must be at least quasi-essential for some point $b$. Let $s$ be a monomial which is not inessential in $h$ at $b$. Thus $us$ is a monomial of $f$ which is not inessential at $b$. However, since $u$ in inessential at $a$ then $us$ is inessential at $a$; thus $us$ is a monomial of $f$ which is not quasi-essential at $a$, and that is absurd. Therefore both $g$ and $h$ are primary at $a$. The tangible value of $f(a)$ is 0, so the product of the tangible values of $g(a)h(a)$ must be 0 as well. Assume that the tangible value of $g(a)$ is $c$; then $f=gh=c^{-1}cgh=c^{-1}gch$. The tangible value of $c^{-1}g(a)$ is clearly 0, and since the tangible value of $f$ is $0$ then so is the tangible value of  $ch(a)$. Thus $c^{-1}g,ch \in \mathbb{P}_a$.\\

Define $\psi : \mathbb{P}_a \rightarrow \mathbb{Q}[x_1,...,x_n]$, by sending $f=\sum \layer{0}{b_I}a_I\lambda^I$ to $\psi(f)=\sum b_Ix^I$.\\

We wish to prove that $\psi(fg)=\psi(f)\psi(g)$. For all $f,g \in \mathbb{P}_a$ we know we can write the polynomials in the form $f=\sum \layer{0}{b_I}a_I\lambda^I$ and
$g=\sum \layer{0}{d_I}a_I\lambda^I$. We also know that the product $fg$ is also in $\mathbb{P}_a$, and thus $fg=\sum \layer{0}{e_I}a_I\lambda^I$. Define $e_I$ to be the sum of products of the form $b_Jd_K$ such that $J+K=I$. If $e_I=0$ then we already proved that the monomial $\layer{0}{0}a_I\lambda^I$ is inessential and should be deleted. $\psi(fg)=\sum e_Ix^I$. Now, $\psi(f)\psi(g) = (\sum b_Ix^I)(\sum d_Ix^I)=\sum e_Ix^I$, as desired.\\

Having seen that $\psi$ is an homomorphism, we now prove it is an isomorphism. Clearly $\psi$ is onto, and we will prove it is also injective. Assume $\psi(f)=\psi(g)=\sum b_Ix^I$. Then $f$ and $g$ both must equal $\sum \layer{0}{b_I}a_I\lambda^I$. Note that this is mainly due to the fact that $f$ and $g$ are regular.\\

Since $\psi$ is a multiplicative isomorphism, and since $\mathbb{Q}[x_1,...,x_n]$ has unique factorization, $\mathbb{P}_a$ has unique factorization as well. Now, let $f$ be any primary polynomial at $a$. If $c$ is the tangible value of $f(a)$, then $c^{-1}f \in \mathbb{P}_a$; $c^{-1}f$ factors uniquely up to multiplication by a constant and clearly so does $f$ as desired.
\end{proof}

\begin{exmp}
$\lambda_1^2 + \lambda_1\lambda_2 + \lambda_2^2$ is irreducible, like $x^2 + xy + y^2$. However,
$$\lambda_1^2 + \layer{\lambda_1\lambda_2}{2} + \lambda_2^2 = (\lambda_1+\lambda_2)^2,$$
as $x^2+2xy+y^2=(x+y)^2$.
\end{exmp}

\subsection{Non-primary polynomials}
Primary polynomials play an important role in our theory. Recall that the essential part at a certain point is a primary polynomial. Thus, at any corner root, the essential part factors uniquely as a polynomial. One might think this will lead to a proof of unique factorization of general polynomials. However this idea fails, as we will now show.\\

Let us focus on polynomials in two variables. The corner root set consists of line segments, and the points where these segments intersect. As we have seen, the monomials which are quasi-essential on the line segment are quasi-essential at the intersection.\\

Assume that the quasi-essential monomials at point $a$ factor into $h_1h_2$, at point $b$ the quasi-essential monomials factor into $g_1g_2$ and at the line segment between $a,b$ the quasi-essential monomials factor into $r_1r_2$. Clearly, $r_1r_2$ are obtained by deleting inessential monomials from $h_1h_2$ and $g_1g_2$. Considering $r_1$ as a part of $h_1$ and $g_1$, one can reconstruct part of the factors of the polynomial by identifying $g_1$ with $h_1$. However, consider the case that $r_1=r_2$. One may identify $g_1$ with $h_2$ and $g_2$ with $h_1$ or $g_i$ with $h_i$, for $i=1,2$. As we will see in the next example, this provides a counterexample to unique factorization.

\begin{exmp}
$$f=f_1f_2=(\lambda_2+\lambda_1+\lambda_1^2+(-1)\lambda_1^3)(\lambda_2+0+\lambda_1^2+(-2)\lambda_1^4)$$
$$g=g_1g_2=(\lambda_2+\lambda_1+\lambda_1^2+(-2)\lambda_1^4)(\lambda_2+0+\lambda_1^2+(-1)\lambda_1^3).$$

We will prove that $f_1,f_2,g_1,g_2$ are irreducible and that $f=g$, and thus unique factorization fails.
\end{exmp}

In the above notation, the points $a,b$ are $a=(0,0),b=(1,2)$. For both $g$ and $f$ at $a$, the quasi-essential monomials are $h_1h_2=(\lambda_2+\lambda_1+\lambda_1^2)(\lambda_2+0+\lambda_1^2)$. At $b$, $g_1g_2=(\lambda_1+\lambda_1^2+(-1)\lambda_1^3)(0+\lambda_1^2+(-2)\lambda_1^4)$. At the line between $a$ and $b$, the quasi-essential monomials are $r_1r_2=(\lambda_2+\lambda_1^2)^2$. The reconstruction of the irreducible factors of the polynomial can yield either $f$ or $g$. Next we will prove that this is a good example (i.e., $f=g$).\\

Recall that two regular polynomials are identical if they are root-equivalent. Therefore, it is enough to show that $f=g$ on the corner roots of $f$ and $g$ in order to prove that $f=g$ everywhere. The corner root set of $f_1f_2$ is the union of the corner root set of $f_1$ and $f_2$. In the following figures, we can see the corner roots of $f_1,g_1$ (solid) and of $f_2,g_2$ (dashed).\\

\includegraphics{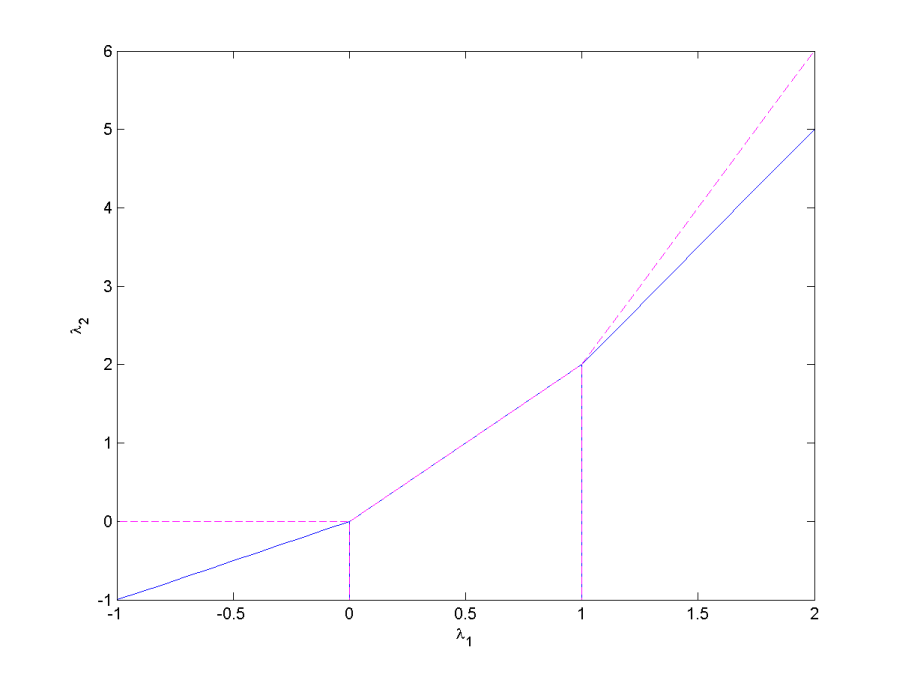}

\includegraphics{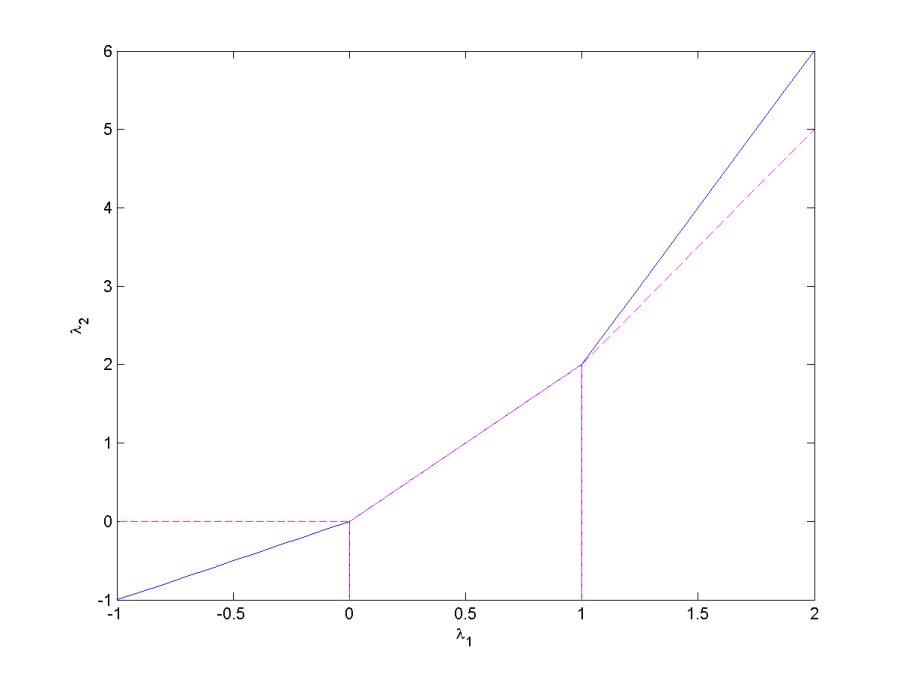}

As we can see in the following table in detail, both $f$ and $g$ have the same quasi-essential monomials at the corner roots set, as desired.\\

\begin{tabular}{|l|l|}
  \hline
  corner root set & quasi-essential monomials \\ \hline
  $\{\lambda_1=\lambda_2<0\}$ & $\lambda_2+\lambda_1$ \\ \hline
  $\{\lambda_2=0,\lambda_1<0\}$ & $\lambda_2(\lambda_2+0)$ \\ \hline
  $\{\lambda_1=\lambda_2=0\}$ & $(\lambda_2+\lambda_1+\lambda_1^2)(\lambda_2+0+\lambda_1^2)$ \\ \hline
  $\{\lambda_1=0,\lambda_2<0\}$ & $(\lambda_1+\lambda_1^2)(0+\lambda_1)$ \\ \hline
  $\{\lambda_1^2=\lambda_2, 0<\lambda_1<1\}$ & $(\lambda_2+\lambda_1^2)(\lambda_2+\lambda_1^2)$ \\ \hline
  $\{\lambda_1=1,\lambda_2=2\}$ & $(\lambda_2+\lambda_1^2+(-1)\lambda_1^3)(\lambda_2+\lambda_1^2+(-2)\lambda_1^4)$ \\ \hline
  $\{\lambda_1=1,\lambda_2<2\}$ & $(\lambda_1^2+(-1)\lambda_1^3)(\lambda_1^2+(-2)\lambda_1^4)$ \\ \hline
  $\{(-1)\lambda_1^3=\lambda_2, 1<\lambda_1\}$ & $(\lambda_2+(-1)\lambda_1^3)((-2)\lambda_1^4)$ \\ \hline
  $\{(-2)\lambda_1^4=\lambda_2, 1<\lambda_1\}$ & $(\lambda_2+(-2)\lambda_1^4)(\lambda_2)$ \\ \hline
\end{tabular}\\

Next we will prove that $f_1,f_2,g_1,g_2$ are irreducible.

\begin{lem}
Let $f$ be a polynomial of the form $f=\lambda_2+g(\lambda_1)$ where $g$ is a polynomial in one variable, then $f$ is irreducible.
\end{lem}
\begin{proof}
Assume that $f=f_1f_2$. Since $\lambda_2$ is a monomial of $f$, without loss of generality, $f_1$ must have $c\lambda_2$ as a monomial, and $f_2$ must have $c^{-1}$ as a monomial for some constant $c$. Clearly $f_1$ and $f_2$ cannot have any monomials of the type $x\lambda_1^n \lambda_2^m$ where $m>1$, since $f$ does not have such monomials. Thus $f_1=h_1(\lambda_1)\lambda_2 + r_1(\lambda_1)$, $f_2=h_2(\lambda_1)\lambda_2 + r_2(\lambda_1)$.\\

Fix $\lambda_1=a$. For any $\lambda_2>r_i(a)h_i^{-1}(a)$, there is a monomial of the form $c\lambda_1^n\lambda_2$ of $f_i$ which is not inessential. Take $\lambda_2>r_1(a)h_1^{-1}(a)+r_2(a)h_2^{-1}(a)$ to obtain a monomial of the form $c\lambda_1^n\lambda_2^2$ of $f$ which is not inessential, in contradiction to the form of $f$.\\

Assume $h$ has a monomial of the form $c\lambda_1^n$ with $n>0$, and assume that this monomial is essential or quasi-essential at a point $b$. Take $\lambda_2>g(b)$, to obtain an essential or quasi-essential monomial of $f$ of the form $c\lambda_1^n\lambda_2$, which is absurd. Thus $f_2=c$ where $c$ is a constant.\\

Therefore $f=cf_1$ is the only factorization of $f$ and thus $f$ is irreducible, as desired.
\end{proof}

To conclude, $f_1,f_2,g_1,g_2$ are irreducible due to the lemma above and so $f_1f_2$ and $g_1g_2$ are two different factorizations of the same polynomial $f= g$.\\

Department of mathematics, Bar-Ilan university, Ramat-Gan 52900, Israel

E-mail address: erez@math.biu.ac.il

\end{document}